\documentclass[secthm,seceqn,amsthm,ussrhead,12pt]{amsart}
\usepackage{amsmath,latexsym}
\usepackage[english]{babel}
\usepackage[psamsfonts]{amssymb}
\usepackage{times}
\usepackage{cite}
\usepackage{color}
\usepackage[mathcal]{euscript}
\numberwithin{equation}{section} \textwidth=15.5cm
\newtheorem{theorem}{Theorem}[section]

\newtheorem{lemma}[theorem]{Lemma}

\newtheorem{example}[theorem]{Example}

\theoremstyle{definition}

\usepackage[colorlinks, bookmarks=true]{hyperref}
\usepackage{color,graphicx,shortvrb}

\usepackage{enumerate}
\usepackage{amsmath}
\usepackage{amssymb}
\usepackage{cite}

\usepackage[latin 1]{inputenc}

\begin{document}

\numberwithin{equation}{section}

\title[Local derivations and Commutativity ]{Local derivations on measurable operators and Commutativity}

\author[Ayupov]{Shavkat Ayupov}
\email{sh$_{-}$ayupov@mail.ru}
\address{Dormon yoli 29, Institute of
 Mathematics,  National University of
Uzbekistan,  100125  Tashkent,   Uzbekistan}

\author[Kudaybergenov]{Karimbergen Kudaybergenov}
\email{karim2006@mail.ru}
\address{Ch. Abdirov 1, Department of Mathematics, Karakalpak State University, Nukus 230113, Uzbekistan}



\date{}
\maketitle

\begin{abstract} We prove that a von Neumann algebra \(M\) is abelian if and only if the square
of every derivation on the algebra $S(M)$ of measurable operators affiliated with \(M\) is a local derivation.
We also show that for general associative unital algebras this is not true.

\end{abstract}

{\it Keywords:} von Neumann algebra, measurable operators,
derivation, local derivation.
\\

{\it AMS Subject Classification:} 17A36,  46L57.

\maketitle \thispagestyle{empty}


\section{Introduction}\label{sec:intro}

 Let
$\mathcal{A}$ be an algebra over the field of complex numbers. A
linear operator $D:\mathcal{A}\rightarrow
\mathcal{A}$ is called  \textit{a derivation} if it satisfies the identity
$D(xy)=D(x)y+xD(y)$ for all  $x, y\in \mathcal{A}$ (Leibniz rule).
Each element  $a\in \mathcal{A}$ defines a derivation $D_a$
on $\mathcal{A}$ given by $D_a(x)=ax-xa,\,x\in \mathcal{A}.$ Such
derivations $D_a$ are said to be \textit{inner}. If the element
$a$ implementing the derivation $D_a$ on $\mathcal{A},$ belongs to
a larger algebra $\mathcal{B},$ containing  $\mathcal{A}$ (as a
proper ideal as usual) then $D_a$ is called a  \textit{spatial
derivation}.

One of the main problems considered in the theory of derivations
is to prove the automatic continuity, innerness or spatialness of
derivations, or to show the existence of non inner and
discontinuous derivations on various topological algebras.

In particular, it is a general algebraic problem to find  algebras
which admit only inner derivations.

Such problems in the framework of algebras of measurable operators
affiliated with von Neumann algebras has been considered in
several papers (e.g. ~\cite{Alb2, AK2011, Ayupov, AK2013,
AKIE13}).

A linear operator $\Delta$ on an algebra $\mathcal{A}$ is called
\textit{local derivation} if given any $x\in \mathcal{A}$ there
exists a derivation $D$ (depending on $x$) such that
$\Delta(x)=D(x).$  The main problem concerning these operators is
to find conditions under which local derivations become
derivations and to present examples of algebras which admit local
derivations that are not derivations (see e.g.~\cite{Kad},
\cite{Lar}). In particular Kadison in ~\cite{Kad} proves that
every continuous local derivation from a von Neumann algebra $M$
into a dual $M$-bimodule is a derivation.
 This theorem gave rise to studies and several results
 on local derivations on C$^\ast$-algebras, culminating with a
definitive contribution due to Johnson ~\cite{John},
who proved that  every (not necessary continuous) local derivation
of a C$^\ast$-algebra  is a derivation. A comprehensive survey of
recent results concerning local  derivations on
C$^\ast$- and von Neumann algebras is presented in ~\cite{AKP}.

In \cite{Nur} we have  studied local derivations on the algebra
$S(M,\tau)$ of all $\tau$-measurable operators affiliated with a
von Neumann algebra $M$ and a faithful normal semi-finite trace
$\tau.$ One of our main results in \cite{Nur} (Theorem~2.1) presents an
unbounded version of Kadison's Theorem A from~\cite{Kad}, and it
asserts that every local derivation on $S(M,\tau)$ which is
continuous in the measure topology automatically becomes a
derivation. In particular, for type I von Neumann algebras $M$ all
such local derivations on $S(M,\tau)$ are inner derivations. We
also proved   that for type I finite von Neumann algebras without
abelian direct summands, as well as for von Neumann algebras with
the atomic lattice of projections, the continuity condition on
local derivations in the above theorem  is redundant. For survey of results concerning
local derivations on algebras of measurable operators we refer to \cite{AK}.

The most interesting effects appear when  we consider the problem of existence of  local
derivations which are not derivations. The consideration of such examples on various  finite-
and infinite dimensional algebras was initiated  by Kadison,
Kaplansky and Jensen (see~\cite{Kad}). We have considered this
problem on a class of commutative regular algebras, which includes
the algebras of measurable operators affiliated with abelian von Neumann algebras.
Unlike C$^\ast$- and von Neumann algebras cases, when all derivations, and hence all local derivations, are trivial,
 the abelian algebras $S(M)$ may admit non-zero derivations (see ~\cite{Ber}).
In ~\cite[Theorem 3.2]{Nur} the authors obtained necessary and
sufficient conditions  for the algebras of measurable and
$\tau$-measurable operators affiliated with a commutative von
Neumann algebra to admit local derivations which are not
derivations. In the latter paper we have proved some properties of
derivations and local derivations which seem to be specific for
the abelian case as it was noted by Professor E.~Zelmanov (UCSD).
In the present paper we confirm this conjecture and give some
criteria for a von Neumann algebra $M$ to be abelian in terms of
properties of derivations and local derivations on the algebra
$S(M)$ of measurable operators affiliated with $M$.

\section{The Main result}\label{sec:main}

The following theorem is the main result of the paper.

\begin{theorem}\label{mainresult}
Let \(S(M)\) be the algebra of measurable operators affiliated
with a von Neumann algebra \(M.\) The following conditions are
equivalent:
\begin{itemize}
\item[(a)]     \(M\) is abelian;
\item[(b)]      For every derivation \(D\) on \(S(M)\) its square \(D^2\) is a local derivation;
\item[(c)]  A linear map \(\Delta: S(M) \to S(M)\) is a local derivation if and only if
\(\Delta(p)=0\) for all projection \(p\) in \(M,\) and
\(s(\Delta(x))\leq s(x)\) for all \(x\)  in \(S(M)\),
 where \(s(x)\) is the support projection of the element \(x.\)
\end{itemize}
\end{theorem}

Let us first present some necessary definitions and facts from the theory of
measurable operators, which  will be used in the proof of this theorem.

Let \(H\) be a Hilbert space over the field \(\mathbb{C}\) of
complex numbers and let \(B(H)\) be the algebra of all bounded
linear operators on \(H.\) Denote by \(\mathbf{1}\) the identity
operator on \(H\) and let \(\mathcal{P}(H) = \{p \in B(H) : p^2 =
p^2 = p^\ast\}\) be the lattice of projections in \(B(H).\)
Consider a von Neumann algebra \(M\) on \(H,\) i.e. a weakly
closed *-subalgebra in \(B(H)\) containing the operator
\(\mathbf{1}\) and denote by \(\|\cdot\|_M\)  the operator norm on
\(M.\) The set \(\mathcal{P}(M) = \mathcal{P}(H)\cap M\) is a
complete orthomodular lattice with respect to the natural partial
order on \(M_h = \{x \in  M: x^\ast = x\},\) generated by the cone
\(M_+\) of positive operators from \(M.\)

Two projections $e,f\in \mathcal{P}(M)$ are said to be
\textit{equivalent} (denoted by $e\sim f$) if there exists a
partial isometry $v\in M$ with initial projection $e$ and  final
projection $f$, i.e. $v^{\ast}v=e,\, vv^{\ast}=f $. The relation
$"\sim"$ is equivalence relation on the lattice $\mathcal{P}(M).$

A linear subspace $\mathcal{D}$ in $H$ is said to be
\textit{affiliated} with $M$ (denoted as $\mathcal{D}\eta M$), if
$u(\mathcal{D})\subset \mathcal{D}$ for every unitary  $u$ in the
commutant
$$M'=\{y\in B(H):xy=yx, \,\forall x\in M\}$$ of the von Neumann algebra $M$ in $B(H)$.

A linear operator  $x$ on  $H$ with the domain  $\mathcal{D}(x)$
is said to be \textit{affiliated} with  $M$ (denoted as  $x\eta
M$) if $\mathcal{D}(x)\eta M$ and $u(x(\xi))=x(u(\xi))$
 for all  $\xi\in
\mathcal{D}(x)$ and for every unitary  $u$ in $M'$.

A linear subspace $\mathcal{D}$ in $H$ is said to be
\textit{strongly dense} in  $H$ with respect to the von Neumann
algebra  $M,$ if
\begin{enumerate}
\item[1)] $\mathcal{D}\eta M;$
\item[2)] there exists a sequence of projections
$\{p_n\}_{n=1}^{\infty}$ in $\mathcal{P}(M)$  such that
$p_n\uparrow\textbf{1},$ $p_n(H)\subset \mathcal{D}$ and
$p^{\perp}_n=\textbf{1}-p_n$ is finite in  $M$ for all
$n\in\mathbb{N}.$
\end{enumerate}

A closed linear operator  $x$ acting in the Hilbert space $H$ is
said to be \textit{measurable} with respect to the von Neumann
algebra  $M,$ if
 $x\eta M$ and $\mathcal{D}(x)$ is strongly dense in  $H.$ Denote by
 $S(M)$ the set of all measurable operators with respect to
 $M.$

It is well-known (see e.g.~\cite{Mur}) that the set $S(M)$  is a
unital *-algebra when equipped with the algebraic operations of
the strong addition and multiplication and taking the adjoint of
an operator.

For every \(x \in  S(M)\)  we set \(s(x) = l(x)\vee r(x),\) where
\(l(x)\) and \(r(x)\) are left and right supports of \(x\)
respectively (see \cite{Mur}).

Let \(\Delta:S(M)\to S(M)\) be a local derivation. Then
\begin{equation}\label{localidem}
p\Delta(p)p=0
\end{equation}
for every idempotent \(p\in M.\)

Indeed, for any idempotent $p\in M$, we have
\[
\Delta(p)=D_p(p)=D(p^2)=D_p(p)p + pD_p(p)= \Delta(p)p +
p\Delta(p).
\]
Thus
\[
\Delta(p)=\Delta(p)p + p\Delta(p)
\]
for every idempotent $p\in M.$ Multiplying both sides of the last
equality by \(p\) we obtain \eqref{localidem}.

Let \(M\) be an abelian von Neumann algebra and let \(D\) be a
derivation on \(S(M).\) By \cite[Proposition 2.3]{Ber} or
\cite[Theorem]{Ayupov}  we have that \(s(D(x)) \leq s(x)\) for any
\(x\in S(M)\) and also \(D(p)=0\) for any  \(p\in
\mathcal{P}(M).\) Therefore by the definition, each local
derivation \(\Delta\) on \(S(M)\) satisfies the following two
conditions:
\begin{equation}\label{firstcon}
s(\Delta(x)) \leq s(x)
\end{equation}
for all \(x\in S(M)\) and
\begin{equation}\label{secondcon}
\Delta|_{\mathcal{P}(M)} =0.
\end{equation}
This means that \eqref{firstcon} and \eqref{secondcon} are
necessary conditions for a linear operator \(\Delta\) to be a
local derivation on the algebra \(S(M).\) We are going to show
that these two condition are in fact also sufficient. The following Lemma
in a more general setting of regular commutative algebras was obtained in \cite[Lemma 3.2]{Nur}.
For the sake of completeness below we an give an alternative and straightforward proof of this result.

\begin{lemma}\label{lemm}
Let \(M\) be an abelian von Neumann algebra. Then each linear
operator \(\Delta\) on the algebra \(S(M)\) satisfying conditions
\eqref{firstcon} and \eqref{secondcon} is a local derivation on
\(S(M).\)
\end{lemma}

\begin{proof} Let us first assume that \(M\) is an abelian von Neumann algebra with a faithful normal
finite trace \(\tau.\) In this case the algebra \(S(M)\) is a
complete metrizable regular algebra with respect to the metric
\(\rho(x,y)=\tau(s(x-y)),\, x, y\in S(M)\) (see \cite[Example
2.6]{Ber}).

It is well-known that the abelian von Neumann algebra \(M\) can be
identified with the algebra \( C(\Omega)\) of all complex valued
continuous functions on a hyperstonean compact space \(\Omega.\)

Let \(a\in S(M)\) be a fixed element. In order to proof that the linear map \(\Delta\) is
 a local derivation we must find a derivation
\(D\) on \(S(M)\) such that \(\Delta(a)=D(a).\)

Let \(P[t]\) be the algebra  of all polynomials at variable \(t\)
over \(\mathbb{C}.\) Consider the following subalgebra of
\(S(M):\)
\[
\mathcal{A}(a)=\left\{f(a): f\in P[t], f(0)=0\right\}.
\]
Consider the linear operator \(D\) from  \(\mathcal{A}(a)\)
into \(S(M)\) defined as
\begin{equation}\label{deri}
D(f(a))=f'(a)\Delta(a),
\end{equation}
where \(f'\) is the  usual derivative of the polynomial \(f.\) It
is clear that \(D\) is a derivation. Let us show that
\begin{equation}\label{supp}
s(D(f(a)))\leq s(f(a))
\end{equation}
for all \(f\in P[t]\) with \(f(0)=0.\)

 Case 1. Suppose first that  $a$ is a bounded element, i.e.  \(a\in M\equiv C(\Omega).\)
Take an arbitrary \(f(t)=\sum\limits_{k=1}^n \lambda_k t^k\in
P[t].\) Denote  \(e=\mathbf{1}-s(f(a)).\) There exists a closed
subset \(S\) of \(\Omega\) such that \(e\) equals to the
characteristic function \(\chi_S\) of the set \(S.\) Since
\(ef(a)=0,\) it follows that \(\sum\limits_{k=0}^n \lambda_k a(t)^k=0\) for all \(t\in S,\)
 that is the complex numbers \(a(t)\) are roots of
\(f\) for all $t\in S.$ Since the polynomial $f$ may have at most  $n$ roots the set \(\left\{a(t): t\in
S\right\} \) is finite. This means that  \(e a\) is a simple (step) function, i.e. it is
a linear combination of projections. Since by \eqref{secondcon} \(\Delta(p)=0\) for all \(p\in
\mathcal{P}(M),\) it follows that \(\Delta(e a)=0.\) Further from the linearity of \(\Delta\)
and properties of the support we have
\begin{eqnarray*}
s\left(\Delta(a)\right) & = &
s\left(\Delta((\mathbf{1}-e)a+ea)\right)=
s\left(\Delta((\mathbf{1}-e)a)+\Delta(ea)\right)= \\
& = & s\left(\Delta((\mathbf{1}-e)a) \right)\leq
s(\mathbf{1}-e)=s(f(a)).
\end{eqnarray*}
Therefore
\begin{eqnarray*}
s\left(D(f(a))\right) & = & s\left(f'(a)\Delta(a)\right)\leq
s\left(\Delta(a)\right) \leq s(f(a)).
\end{eqnarray*}

 Case 2. Let \(a\in S(M)\) be an arbitrary element. There exists an increasing
 sequence of projections \(\{e_n\}\) in \(M\) such that \(e_n a\in
 M\) for all \(n\in \mathbb{N}\) and \(\bigvee\limits_{n\in
 \mathbb{N}}e_n=\mathbf{1}.\) Taking into account  the equalities
 \(e_nD(x)=D(e_n x),\)  \(e_n f(a)=f(e_n a)\) (the second equality follows from \(f(0)=0\)) and the Case 1, we obtain that
\begin{eqnarray*}
e_n s\left(D(f(a))\right) & = & s\left(e_n
D(f(a))\right)=s\left(D(e_n f(a))\right)=s\left(D(f(e_n
a))\right)\leq \\
& \leq & s\left(f(e_n a)\right) = s(e_n f(a))=e_n s(f(a))
\end{eqnarray*}
for all \(n.\) Thus
\begin{eqnarray*}
s\left(D(f(a))\right) \leq s(f(a)).
\end{eqnarray*}
Thus, the derivation \(D\) defined by \eqref{deri} satisfies the
condition \eqref{supp}. Since \(S(M)\) is a complete metrizable
regular algebra,  by \cite[Theorem 3.1]{Ber} the derivation \(D\)
from \(\mathcal{A}(a)\) into \(S(M)\) can be extended  to a
derivation on the whole algebra \(S(M),\) which we also denote by
\(D.\) Taking the polynomial \(p(t)=t\) in the the definition
\eqref{deri} of \(D\), we obtain that \(\Delta(a)=D(a).\) This
means that \(\Delta\) is a local derivation.

Now let \(M\) be an arbitrary abelian von Neumann algebra. It is well-known
that it has a faithful normal semifinite trace  $\tau.$ There
exists a family of mutually orthogonal projections \(\{z_i\}_{i\in
I}\) in \(M\) such that \(\bigvee\limits_{i\in I} z_i
=\mathbf{1}\) and \(\tau(z_i)<+\infty\) for all \(i\in I.\)

Take \(x=z_i x\in z_i S(M)\equiv S(z_i M).\) The condition
\eqref{firstcon} implies that
\begin{eqnarray*}
s(\Delta(x))=s(\Delta(z_i x))\leq s(z_i x)\leq z_i.
\end{eqnarray*}
Thus \(\Delta(x)=z_i\Delta(x)\in S(z_i M).\) This means that
\(\Delta\) maps \(S(z_i M)\) into itself for all \(i\in I.\)
Therefore, since  \(z_i M\) is an abelian von Neumann algebra with
a faithful normal finite trace, from above it follows that  the
restriction \(\Delta|_{S(z_i M)}\) of \(\Delta\) onto \(S(z_i M)\)
is a local derivation. Thus \(\Delta\) is also a local derivation.
The proof is complete.
\end{proof}

\begin{proof}[\textit{Proof of Theorem~$\ref{mainresult}$}]

(a) \(\Rightarrow\) (c). Let \(M\) be abelian. Before
Lemma~\ref{lemm} we already mentioned that any local derivation on
\(S(M)\) satisfies the conditions \eqref{firstcon} and
\eqref{secondcon}.

Conversely, suppose that \(\Delta\) is a linear operator on \(S(M)\) such
that \(s(\Delta(x)) \leq s(x)\) for any \(x\in S(M)\)  and
\(\Delta(p) = 0\) for all \(p\in \mathcal{P}(M).\) By
Lemma~\ref{lemm} it follows  that \(\Delta\) is a local
derivation.

(c) \(\Rightarrow\) (b). Let  \(D\) be a derivation on \(S(M).\)
Then \(D\) is a local derivation. By the assumption (c) we have
that
\[
s(D^2(x))=s(D(D(x))\leq s(D(x))\leq s(x)
\] for all \(x\)  in \(S(M)\) and
\[
D^2(p)=D(D(p))=D(0)=0
\]
for all projection \(p\) in \(M.\) Applying (c) to a linear
operator  \(D^2,\) we conclude that it is a local derivation.

(b) \(\Rightarrow\) (a). Suppose that \(M\) be a noncommutative
von Neumann algebra and \(a\in M\) be a non central  element. By
the assumption of Theorem a linear operator \(\Delta:S(M)\to
S(M)\) defined by
\[
\Delta(x)=[a, [a, x]],\, x\in S(M)
\]
is a local derivation. By \eqref{localidem}, we obtain that
\[
p[a,[a,p]]p=0
\]
for every projection \(p\in M.\) Thus
\[
p(a^2 p+pa^2-2apa)p=0,
\]
i.e.,
\begin{equation}\label{aaa}
pa^2 p=papap.
\end{equation}

Since \(M\) is noncommutative, there exists a non central
projection \(e\in M.\) Then \linebreak \(z(e)z(\mathbf{1}-e)\neq
0,\) where \(z(x)\) denotes the central support of the element \(x.\) By
\cite[Proposition 6.1.8]{KadisonII} we can find mutually
orthogonal, equivalent nonzero projections \(p, q\in M\) such that
\(p\leq e\) and \(q\leq\mathbf{1}-e.\) Take a partial isometry
\(u\in M\) such that \(uu^\ast=p\) and \(u^\ast u=q.\) Let us
consider the element
\[
a=u+u^\ast.
\]
Using the equality \(u=u u^\ast u,\) we obtain that
\[
u^2=u u= (u u^\ast u) (u u^\ast u)=u (u^\ast u) (u u^\ast) u=u q p
u=0.
\]
Then
\[
a^2=uu^\ast +u^\ast u=p+q,
\]
and therefore
\begin{equation}\label{bbb}
pa^2 p=p.
\end{equation}

Further
\[
pap=uu^\ast(u+u^\ast)uu^\ast=uu^\ast u uu^\ast +uu^\ast u^\ast
uu^\ast=0,
\]
and therefore
\begin{equation}\label{ccc}
papa p=0.
\end{equation}
Combining \eqref{bbb}, \eqref{ccc} and \eqref{aaa}, we obtain a
contradiction. This contradiction implies  that \(M\) is
commutative. The proof is complete.
\end{proof}

\textit{Remark}.
In \cite{Ber} the authors have proved that if $M$ is an abelian  von Neumann algebra with a non-atomic lattice of projections $\mathcal{P}(M)$,
then the algebra $S(M)$ admits non-zero derivations. If $D$ is such a derivation then from the above theorem we have that $D^2$ is
a non-zero local derivation. Moreover $D^2$ is not a derivation, because from \cite[Lemma 3.3]{Nur} it follows that $D^2$ is a derivation
if and only if $D$ is identically zero.

\section{Counterexample for general associative algebras}

The following example shows that Theorem~\ref{mainresult} and the last Remark fail
in the general case  of unital associative algebras.

\begin{example} \end{example}
Let \(T_2(\mathbb{C})\) be the algebra of all upper triangular
\(2\times 2\) matrices  over \(\mathbb{C},\) i.e.,
\[
T_2(\mathbb{C})=\left\{\left(%
\begin{array}{cc}
  \lambda_{11} & \lambda_{12} \\
  0 & \lambda_{22} \\
\end{array}%
\right): \lambda_{ij}\in \mathbb{C},\, 1\leq i\leq j\leq
2\right\}.
\]
Let   \(e_{11}, e_{12}, e_{22}\)  be the matrix units in
\(T_2(\mathbb{C}).\) It is known \cite{Coelho} that any derivation
on the algebra \(T_2(\mathbb{C})\) is inner. Direct computations
show that the inner derivation generated by an element
\(a=\sum\limits_{1\leq i\leq j}^2 a_{ij} e_{ij}\in
T_2(\mathbb{C})\) acts on \(T_2(\mathbb{C})\) as
\[
D\left(\sum\limits_{1\leq i\leq j}^2 \lambda_{ij}
e_{ij}\right)=\left[(a_{11}-a_{22})\lambda_{12}-a_{12}(\lambda_{11}-\lambda_{22})\right]e_{12}.
\]
Thus
\[
D^2\left(\sum\limits_{1\leq i\leq j}^2 \lambda_{ij}
e_{ij}\right)=\left[(a_{11}-a_{22})^2\lambda_{12}-(a_{11}-a_{22})a_{12}(\lambda_{11}-\lambda_{22})\right]e_{12}
\]
and \(D^2\) is an inner derivation generated by the matrix
\[
\left(%
\begin{array}{cc}
  a_{11}^2+a_{22}^2 & (a_{11}-a_{22})a_{12} \\
  0 & 2a_{11}a_{22} \\
\end{array}%
\right).
\]
So, the square of every derivation on the algebra
\(T_2(\mathbb{C})\)  is a derivation (and hence a local derivation), but \(T_2(\mathbb{C})\) is
non commutative.

\section*{Acknowledgements}

The idea of this paper came during the visit of the first author to the University of California, San Diego.
The first author would like to thank Professor Efim Zelmanov for hospitality and useful discussions during the stay  in UCSD.

\end{document}